\documentclass{amsart}
\usepackage{braket} 
\usepackage{mathtools}

	\theoremstyle{plain}
	\newtheorem{thm}{Theorem}
	\numberwithin{equation}{section}
	\providecommand{\SetS}[1]{ \left\{ #1 \right\} }
	\providecommand{\LIP}[2]{ \left \langle #1 , #2 \right \rangle}
	\providecommand{\WholeSpace}{\mathbb{R}^4}
	\providecommand{\HyperboloidUpper}{H_T^{+}}
	\providecommand{\HyperboloidOne}{H_S}
	\providecommand{\LightCone}{L^{+}}
	\providecommand{\Ball}{B^{3}}
	\providecommand{\Disc}{B^{2}}
	\providecommand{\Vector}[1]{\boldsymbol{#1}}
	\providecommand{\HalfSpace}[1]{R_{#1}}
	\providecommand{\Plane}[1]{P_{#1}}
	\providecommand{\Orthoscheme}[1]{R_{r , \theta} (#1)}
	\providecommand{\VolumeFunction}{V_{r , \theta}}
	\providecommand{\Projection}{\mathcal{P}}
	\providecommand{\Vertex}[1]{P_{#1}}
	\providecommand{\Edge}[2]{P_{#1} P_{#2}}
	\providecommand{\Face}[3]{P_{#1} P_{#2} P_{#3}}
	\providecommand{\VertexLift}[1]{\Vector{p}_{#1}}
	\providecommand{\FaceNormal}[1]{\Vector{u}_{#1}}
	\providecommand{\EdgeLength}[2]{\ell_{#1 , #2}}
	\providecommand{\DihedralAngle}[2]{\theta_{#1 , #2}}
	\DeclareMathOperator{\Arcsinh}{arcsinh}
	\DeclareMathOperator{\Arccosh}{arccosh}
	\providecommand{\RATT}[2]{R_{#1} (#2)}
	\providecommand{\RAT}[1]{\RATT{r}{#1}}
	\providecommand{\AreaFunction}[1]{A_{r} (#1)}
	\providecommand{\AngleAtZero}[1]{a (#1)}
\begin{document}

\title[On the maximal volume of hyperbolic complete orthoschemes]{On the maximal volume of three-dimensional hyperbolic complete orthoschemes}

\author{Kazuhiro Ichihara}

\address{Department of Mathematics, College of Humanities and Sciences, Nihon University, 3-25-40 Sakurajosui, Setagaya-ku, Tokyo 156-8550, Japan}

\email{ichihara@math.chs.nihon-u.ac.jp}

\author{Akira Ushijima}

\address{Faculty of Mathematics and Physics, Institute of Science and Engineering, Kanazawa University, Kanazawa 920--1192, Japan}

\email{ushijima@se.kanazawa-u.ac.jp}

\thanks{The first author is partially supported by JSPS KAKENHI Grant Number 23740061 and Joint Research Grant of Institute of Natural Sciences at Nihon University 2013.
The second author is partially supported by JSPS KAKENHI Grant Number 24540071.}

\date{\today}

\begin{abstract}
	A three-dimensional orthoscheme
	is defined as a tetrahedron
	whose base is a right-angled triangle
	and an edge joining the apex
	and a non-right-angled vertex
	is perpendicular to the base.
	A generalization, called complete orthoschemes,
	of orthoschemes is known
	in hyperbolic geometry.
	Roughly speaking,
	complete orthoschemes consist
	of three kinds of polyhedra;
	either compact, ideal or truncated.
	We consider a particular
	family of hyperbolic complete orthoschemes,
	which share the same base.
	They are parametrized
	by the ``height",
	which represents
	how far the apex is
	from the base.
	We prove
	that the volume attains maximal
	when the apex is ultraideal
	in the sense of hyperbolic geometry,
	and that
	such a complete orthoscheme
	is unique in the family.
\end{abstract}
\maketitle

\section{Introduction} \label{sec_Intro}

In \cite{ke},
Kellerhals wrote
``the most basic objects in polyhedral geometry
are orthoschemes",
and she gave a formula
to calculate the volumes of
complete orthoschemes
in the three-dimensional hyperbolic space.
What we discuss here
is
the existence and the uniqueness
of the maximal volume
of a family of complete orthoschemes
parametrized by the ``height".

Consider a family of pyramids in Euclidean space
with a fixed base polygon and the locus of apexes perpendicular to the base polygon.
The volumes of pyramids
strictly increases
when the height increases,
because pyramids strictly increases as a set.
By the same reason,
this phenomenon holds true for
such a family of pyramids in hyperbolic space.
In contrast to the Euclidean case,
the volume approaches to a finite value.
Furthermore,
in hyperbolic space the apex can ``run out" the space.
Then we can still obtain finite volume hyperbolic polyhedron
by {\em truncation}\/ with respect to the apex.
The volume converges to zero
as the vertex goes away from the space.
So it is an interesting question when the volume becomes maximum.

As is mentioned above,
one of the most fundamental one among all such pyramids  is
the orthoscheme.
An orthoscheme
is a kind of simplex
which has particular orthogonality
among its faces.
Let 
$\Vertex{0}$,
$\Vertex{1}$,
$\Vertex{2}$ and
$\Vertex{3}$
be the vertices of a simplex
$R$
in the three-dimensional hyperbolic space.
We denote by $\Edge{i}{j}$
the edge spanned by $\Vertex{i}$ and $\Vertex{j}$,
and by $\Face{i}{j}{k}$
the face spanned by $\Vertex{i}$, $\Vertex{j}$ and $\Vertex{k}$.
Such a simplex $R$
is called an  {\em orthoscheme}\/
(in the ordinary sense)
if 
the edge $\Edge{0}{1}$
is perpendicular to
the face $\Face{1}{2}{3}$
and the face $\Face{0}{1}{2}$
is orthogonal to $\Edge{2}{3}$.
In other words,
an orthoscheme
is a tetrahedron with
a right-angled triangle $\Face{0}{1}{2}$
as its base
and an edge joining the apex
and a non-right-angled vertex, say $\Vertex{2}$,
is perpendicular to the base.
Vertices $\Vertex{0}$ and $\Vertex{3}$
are called the {\em principal vertices}\/ of $R$.
Its precise definition
will be given in Section~\ref{sec: orthoschemes}.

Though orthoschemes are also considered
in Euclidean or
spherical spaces, 
in hyperbolic space
the ordinary orthoschemes
are extended to the so-called
{\em complete orthoschemes}.
Let $\Ball$
be the open unit ball in the three-dimensional
Euclidean space $\mathbb{R}^3$
centered at the origin.
The set $\Ball$
can be regarded as the so-called
{\em projective ball model}\/
of the three-dimensional hyperbolic space.
Any tetrahedron in hyperbolic space
appears as a Euclidean tetrahedron in $\Ball$.
If
one or both principal vertices 
of an orthoscheme $R$ lie in
the boundary of $\Ball$,
the set $R \cap \Ball$
is called an {\em ideal polyhedron},
which is not bounded
in hyperbolic space,
while its volume is finite.
Take one step further
and
we allow
principal vertices
to be in the exterior of $\Ball$.
The volume of $R \cap \Ball$
is no longer finite,
but there is a canonical way
to delete ends of 
$R \cap \Ball$ with infinite volume
so that
we obtain a polyhedron of finite volume,
called a {\em truncated polyhedron}.
Complete orthoschemes
are, roughly speaking,
either compact, ideal 
or truncated orthoschemes.
The precise definitions
of complete orthoschemes and
truncation
will also be given in Section~\ref{sec: orthoschemes}.

What we study in this paper
is
the maximal volume
of a
family of complete orthoschemes
with one parameter.
Consider a
family of complete orthoschemes
that share
the same base
$\Face{0}{1}{2}$.
We allow the vertex $\Vertex{0}$
to be in the exterior of $\Ball$.
In this case the base
$\Face{0}{1}{2}$
means the truncated polygon
obtained from
the triangle
with vertices $\Vertex{0}$,
$\Vertex{1}$ and $\Vertex{2}$.
Such a
family
of complete orthoschemes
is
parametrized by
the hyperbolic length
of the edge $\Edge{2}{3}$
when $\Vertex{3}$ is in $\Ball$.
When the hyperbolic length
increases,
the orthoscheme strictly
increases as a set,
which means the volume
also increases with respect to
the function of the hyperbolic length.
This phenomenon
holds until
the vertex $\Vertex{3}$
lies in
the boundary $\partial \Ball$ of $\Ball$.
The hyperbolic length of $\Edge{2}{3}$
is ``beyond" the infinity
when $\Vertex{3}$
is in the exterior of $\Ball$,
but we have a complete
orthoscheme
with finite volume
by
truncation.
Instead of the hyperbolic length,
using
the Euclidean length
of $\Edge{2}{3}$,
which we mentioned as ``height" 
in the first paragraph,
we
can parametrize the family
even if $\Vertex{3}$ is in the complement of $\Ball$.
The complete orthoscheme
approaches the empty set
when
$\Vertex{3}$
goes far away from 
$\Ball$.
The family thus has
maximal
volume
complete orthoschemes,
which arise when $\Vertex{3}$
lies
in the complement of $\Ball$.

As a toy model,
let us consider the
same
phenomenon
for
the two-dimensional orthoschemes, namely 
hyperbolic triangle $\Face{0}{1}{2}$
with right angle at $\Vertex{1}$.
Take a family of
complete orthoschemes
parametrized by the ``height" of
$\Edge{1}{2}$.
The
area
strictly increases
when $\Vertex{2}$
approaches to the boundary of $\Disc$,
the projective disc model of the two-dimensional hyperbolic space.
The area attains maximal when $\Vertex{2}$ lies in $\partial \Disc$.
When $\Vertex{2}$ is in the exterior of $\Disc$,
the area decreases, but not necessarily monotonically.
These facts are summarized as Theorem~\ref{thm: 2ortho}
in the appendix.

One may expect that
the same phenomenon happens
for three-dimensional complete orthoschemes.
Is the volume attains maximal
at least
when $\Vertex{3}$ is in $\partial \Ball$?
Does the volume decrease when $\Vertex{3}$
goes far away from $\Ball$?
Our main result,
which is Theorem~\ref{thm: main} in 
Section~\ref{sec: MainResult},
answers both of the questions negatively.

\section{Preliminaries of hyperbolic geometry} \label{sec: preliminaries}
There are several models to
introduce hyperbolic geometry.
Among them
we use
the {\em hyperboloid model}\/
to calculate lengths and angles with respect to the hyperbolic metric,
and use the {\em projective ball model}\/
to define complete orthoschemes.
Definitions of these two models,
together with formulae to calculate
hyperbolic lengths and hyperbolic angles,
are explained in this section.
See \cite{ra}
for basic references on hyperbolic geometry.

As a set,
the hyperboloid model
$\HyperboloidUpper$
of the
three-dimensional
hyperbolic space
is defined
as a subset of the four-dimensional
Euclidean space $\WholeSpace$
by
\begin{equation*}
	\HyperboloidUpper \vcentcolon
	= \Set{ \Vector{x} = (x_0 , x_1 , x_2 , x_3) \in \WholeSpace | \LIP{\Vector{x}}{\Vector{x}} = -1 \text{ and } x_0 > 0 } ,
\end{equation*}
where $\LIP{\cdot}{\cdot}$, called the{\em Lorentzian inner product}, is
defined as
\begin{equation*}
	\LIP{\Vector{x}}{\Vector{y}} \vcentcolon = - x_0 \, y_0 + x_1 \, y_1 + x_2 \, y_2 + x_3 \, y_3 + x_4 \, y_ 4
\end{equation*}
for any $\Vector{x} = (x_0 , x_1 , x_2 , x_3)$ and $\Vector{y} = (y_0 , y_1 , y_2 , y_3)$ in $\WholeSpace$.
The restriction of the quadratic form induced from the Lorentzian inner product
to the tangent spaces of $\HyperboloidUpper$
is positive definite
and gives a Riemannian metric on $\HyperboloidUpper$,
which is constant curvature of
$-1$.
The set $\HyperboloidUpper$ together with this metric
gives
the {\em hyperboloid model}
of the three-dimensional hyperbolic space.

Associated with
$\HyperboloidUpper$,
there are two important subsets of $\WholeSpace$:
\begin{align*}
	\HyperboloidOne & \vcentcolon = \Set{ \Vector{x} \in \WholeSpace | \LIP{\Vector{x}}{\Vector{x}} = 1 } ,&
	\LightCone & \vcentcolon= \Set{ \Vector{x} \in \WholeSpace | \LIP{\Vector{x}}{\Vector{x}} = 0 \text{ and } x_0 > 0 } .
\end{align*}
Every point $\Vector{u}$ in $\HyperboloidOne$
corresponds to a half-space
\begin{equation*}
	\HalfSpace{\Vector{u}} \vcentcolon =  \Set{ \Vector{x} \in \WholeSpace | \LIP{\Vector{x}}{\Vector{u}} \leq 0 } ,
\end{equation*}
bounded by a plane
\begin{equation*}
	\Plane{\Vector{u}} \vcentcolon =  \Set{ \Vector{x} \in \WholeSpace | \LIP{\Vector{x}}{\Vector{u}} = 0 } .
\end{equation*}
The intersection $\Plane{\Vector{u}} \cap \HyperboloidUpper$
is a geodesic plane with respect to the hyperbolic metric.
If
$\Vector{u}$ is taken from $\LightCone$,
the set $\HalfSpace{\Vector{u}}$
is defined as
\begin{equation*}
	\HalfSpace{\Vector{u}} \vcentcolon =  \Set{ \Vector{x} \in \WholeSpace | \LIP{\Vector{x}}{\Vector{u}} \leq - \frac{1}{2} } .
\end{equation*}
The intersection
$\HalfSpace{\Vector{u}} \cap \HyperboloidUpper$ is called a {\em horoball}.
The intersection of the boundary
\begin{equation*}
	\Plane{\Vector{u}} \vcentcolon =  \Set{ \Vector{x} \in \WholeSpace | \LIP{\Vector{x}}{\Vector{u}} = - \frac{1}{2} } 
\end{equation*}
of $\HalfSpace{\Vector{u}}$
and $\HyperboloidUpper$
is called a {\em horosphere}.

The Lorentzian inner product
is also used to calculate
distances and angles with respect to the hyperbolic metric.
The details of the following results
are explained
in {\S}3.2 of \cite{ra}.
Let $\Vector{u}$ be a point in $\HyperboloidUpper$
and let $\Vector{v}$ be taken from $\HyperboloidOne$ with $\Vector{u} \in \HalfSpace{\Vector{v}}$,
then the hyperbolic distance $\ell$
between $\Vector{u}$ and the geodesic plane
$\Plane{\Vector{v}}$ is
calculated by 
\begin{equation} \label{eq: DistancePointPlane}
	\sinh \ell  = - \LIP{\Vector{u}}{\Vector{v}} .
\end{equation}
Suppose that $\Vector{u}$ is in $\LightCone$
and $\Vector{v}$ is in $\HyperboloidOne$ with $\Vector{u} \in \HalfSpace{\Vector{v}}$.
Let $\ell$
be the signed hyperbolic distance
between
the horosphere $\Plane{\Vector{u}} \cap \HyperboloidUpper$
and the geodesic plane $\Plane{\Vector{v}}  \cap \HyperboloidUpper$.
The sign is defined to be positive
if
the horosphere and the geodesic plane
do not intersect,
otherwise negative.
Then 
the signed distance
$\ell$
is calculated by
\begin{equation} \label{eq: DistancePointHorosphere}
	\frac{e^{\ell}}{2}  = - \LIP{\Vector{u}}{\Vector{v}} .
\end{equation}
If
both $\Vector{u}$ and $\Vector{v}$
are taken from $\HyperboloidOne$
with $\Vector{u} \in \HalfSpace{\Vector{v}}$ and $\Vector{v} \in \HalfSpace{\Vector{u}}$,
then there are three
possibilities:
$\HalfSpace{\Vector{u}} \cap \HalfSpace{\Vector{v}}$
intersects $\HyperboloidUpper$,
intersects $\LightCone$
or does not intersect both $\HyperboloidUpper$ and $\LightCone$.
The
first case
means that
the geodesic planes
$\Plane{\Vector{u}} \cap \HyperboloidUpper$ and
$\Plane{\Vector{v}} \cap \HyperboloidUpper$
intersect and form
a corner $\HalfSpace{\Vector{u}} \cap \HalfSpace{\Vector{v}} \cap \HyperboloidUpper$.
The hyperbolic dihedral angle $\theta$
between 
these geodesic planes measured in this corner
is calculated by 
\begin{equation} \label{eq: DihedralAngle}
	\cos \theta = - \LIP{\Vector{u}}{\Vector{v}} .
\end{equation}
The third case
means that
the geodesic planes
$\Plane{\Vector{u}} \cap \HyperboloidUpper$ and
$\Plane{\Vector{v}} \cap \HyperboloidUpper$
are
{\em ultraparallel},
meaning that
they
do not intersect
in $\HyperboloidUpper$
and
there
is
a
unique geodesic line in $\HyperboloidUpper$
which is perpendicular to these geodesic planes.
The hyperbolic length
$\ell$
of the segment between
these planes
is calculated by 
\begin{equation} \label{eq: DistancePlanes}
	\cosh \ell = - \LIP{\Vector{u}}{\Vector{v}} .
\end{equation}
The second case
is regarded as
the first case with hyperbolic dihedral angle $0$
or the third case with the hyperbolic distance $0$.
Geodesic planes in this case
are called {\em parallel}\/ in hyperbolic space.

The {\em projective ball model} $\Ball$
is another model of 
the three-dimensional hyperbolic space,
which is induced
from $\HyperboloidUpper$.
Let
$\Projection$
be the radial projection from
$\WholeSpace - \Set{\Vector{x} \in \WholeSpace | x_0 = 0}$
to the affine hyperplane
$\boldsymbol{P}_1 \vcentcolon = \Set{\Vector{x} \in \WholeSpace | x_0 = 1}$
along the ray from the origin
$\Vector{o}$ of $\WholeSpace$.
The projection $\Projection$
is a homeomorphism on
$\HyperboloidUpper$
to the three-dimensional open unit ball
$\Ball$ in
$\boldsymbol{P}_1$
centered at
$(1,0,0,0)$.
A metric
is
induced
on $\Ball$ 
from $\HyperboloidUpper$
by the projection.
With this metric
$\Ball$
is called the {\em projective ball model}\/
of the three-dimensional hyperbolic space.
The projection $\Projection$
also induces the mapping
from $\WholeSpace - \SetS{\Vector{o}}$
to the three-dimensional real projective space
$\mathbb{R}P^3$,
which is defined to be the union
of $\boldsymbol{P}_1$
and
the set of lines in the affine hyperplane
$\Set{\Vector{x} \in \WholeSpace | x_0 = 0}$
through $\Vector{o}$.
In contrast to ordinary points
in $\Ball$,
points in the set $\partial \Ball$
of the boundary of $\Ball$
are called {\em ideal},
and points in the exterior of $\Ball$
are called {\em ultraideal}.
We often regard $\Ball$
as the unit open ball centered at the origin
in $\mathbb{R}^3$.

We mention important properties of $\Ball$
to be used in the definition of complete orthoschemes
in the next section.
First,
every geodesic plane in $\Ball$
is given as the intersection of
a Euclidean plane and $\Ball$.
This is because
every geodesic plane in $\HyperboloidUpper$
is defined as the intersection of $\HyperboloidUpper$
and a linear subspace of $\WholeSpace$ of dimension three,
and a geodesic plane in $\Ball$
is the image of that in $\HyperboloidUpper$
by the radial projection.
The projection $\Projection$
thus gives a correspondence
between points in 
the exterior of $\Ball$ in $\mathbb{R}P^3$
and the geodesic planes of $\Ball$.
We call
$\Projection (\Vector{u})$
for $\Vector{u} \in \HyperboloidOne$
the {\em pole}\/
of the plane $\Projection (\Plane{\Vector{u}})$
or the geodesic plane $\Projection (\Plane{\Vector{u}} \cap \HyperboloidUpper)$.
Conversely,
we call $\Projection (\Plane{\Vector{u}})$
the {\em polar plane}\/ of
$\Projection (\Vector{u})$,
and we call
$\Projection (\Plane{\Vector{u}} \cap \HyperboloidUpper)$
the {\em polar geodesic planes}\/ of $\Projection (\Vector{u})$.
If $\Projection (\Plane{\Vector{u}} \cap \HyperboloidUpper)$
does not pass through the origin of $\Ball$,
then 
its pole is given as the apex
of a circular cone
which is tangent to $\partial \Ball$
and has the base circle
$\Plane{\Vector{u}} \cap \partial \Ball$.
The second important property
is that,
for a given geodesic plane, say $P$, in $\Ball$,
every plane or line which passes through the pole
of $P$ is orthogonal to $P$ in $\Ball$.
This is proved
by using Equation~\eqref{eq: DihedralAngle}.

\section{Complete orthoschemes} \label{sec: orthoschemes}
Following \cite{ke}
we introduce complete orthoschemes.
As is mentioned in the introduction,
an (ordinary) {\em orthoscheme}\/
in the three-dimensional hyperbolic space
is a tetrahedron
with vertices $\Vertex{0}$, $\Vertex{1}$, $\Vertex{2}$ and $\Vertex{3}$
which satisfies that
$\Edge{0}{1}$
is perpendicular to $\Face{1}{2}{3}$
and that $\Face{0}{1}{2}$
is orthogonal to $\Edge{2}{3}$.
The vertices $\Vertex{0}$ and $\Vertex{3}$ are
called {\em principal vertices}.

{\em Complete orthoschemes}\/
are a generalization of ordinary
orthoschemes
by allowing one or both principal vertices
to be ideal or ultraideal.
Take $\Ball$ as our favorite model
of the hyperbolic space
in what follows.
As a set,
any orthoscheme in the ordinary sense
are given
as a Euclidean tetrahedron in $\Ball$.
When
one or both principal vertices
are ideal,
the tetrahedron
as a set in the hyperbolic space
is no more bounded,
but still has finite volume.
We allow to call such tetrahedra
ordinary orthoschemes.

Further generalization of orthoschemes
is explained via
{\em truncation}\/
of ultraideal
vertices.
Suppose
a vertex $v$ of a tetrahedron $R$
is ultraideal.
Let $T$
be the half-space
bounded by the polar plane of $v$
with $v \not \in T$.
{\em Truncation}\/
of $R$
with respect to
$v$
is defined
as an operation
to obtain a polyhedron $R \cap T$.
If $v$ is
close enough 
to
$\partial \Ball$,
then
$R \cap T$ is non-empty.

Truncation
is also explained
by using
the hyperboloid model.
The inverse image of $v$ for $\Projection$
on $\HyperboloidOne$
consists of two points.
Each of them gives a half-space in $\WholeSpace$,
and one of them corresponds to
the inverse image of $T$.
In this sense
there is a one-to-one
correspondence
between half-spaces in $\Ball$
and points in $\HyperboloidOne$.
The
point
in $\HyperboloidOne$
corresponding
to $v$ with respect to $T$
in the sense above
is called the {\em proper}\/ inverse image of $v$
for truncation of $R$.
This correspondence
will be used 
to calculate hyperbolic lengths of edges
and hyperbolic dihedral angles between faces
of complete orthoschemes.

When
one of the principal vertices,
say $\Vertex{3}$,
is ultraideal
and $\Vertex{0}$ is not (i.e, ordinal or ideal),
we have 
a polyhedron
with finite volume
by
truncation with respect to $\Vertex{3}$.
Such a polyhedron is called
a {\em simple frustum}\/ with ultraideal vertex $\Vertex{3}$.
We remark that
the vertices $\Vertex{0}$, $\Vertex{1}$ and
$\Vertex{2}$
are
simultaneously deleted by
truncation
when
$\Vertex{3}$ is far away from $\Ball$,
since
both the polar geodesic plane of $\Vertex{3}$
and the triangle $\Face{0}{1}{2}$
are orthogonal to $\Edge{2}{3}$
in $\Ball$.

Suppose
both $\Vertex{0}$ and $\Vertex{3}$
are
ultraideal.
There
are
three possibilities:
the polar planes of $\Vertex{0}$ and $\Vertex{3}$
intersect in $\Ball$, they are parallel,
or they are ultraparallel.
In
the first case,
the polyhedron 
we obtain by truncation
is
well known as a {\em Lambert cube}.
See \cite[Figure~2]{ke} for example.
The edge $\Edge{0}{3}$
is deleted by truncation.
In the third case, on the other hand,
the
polyhedron
obtained
by
truncation
still
has
the edge
induced from $\Edge{0}{3}$.
We call this polyhedron
a {\em double frustum}.
The second case is the limiting situation of
both the first and the third cases.
We call polyhedra obtained
in the second case
{\em double frustum with an ideal vertex}.

As a summary,
combinatorial types of complete orthoschemes are either
\begin{itemize}
	\item
		ordinary orthoschemes,
		whose principal vertices are either ordinarily points
		or ideal points,
	\item
		simple frustums,
	\item
		double frustums possibly with an ideal vertex, or
	\item
		Lambert cubes.
\end{itemize}

\section{The Schl\"afli differential formula} \label{sec: Schlafli}

Kellerhals
obtained formulae
to calculate volumes of complete orthoschemes
in \cite{ke};
the formula
for Lambert cubes
is given
in Theorem~III,
and the formula for other kinds of complete orthoschemes
is given in Theorem~II.
In both formulae,
they are parametrized
by
the
three
non-right hyperbolic dihedral angles.
Under the same setting used in Section~\ref{sec: orthoschemes},
we denote by $\DihedralAngle{i}{j}$
the hyperbolic dihedral angle
between
faces opposite to $\Vertex{i}$ and $\Vertex{j}$.
When
a complete orthoscheme
is a Lambert cube,
the
geodesic planes
containing faces opposite to 
$\Vertex{1}$ and $\Vertex{2}$
are
ultraparallel.
In this case
$\DihedralAngle{1}{2}$
is defined to be the hyperbolic dihedral angle
between
the polar geodesic planes of
the vertex $\Vertex{0}$
and $\Vertex{3}$.
In this sense
the formulae
are parametrized by
$\DihedralAngle{0}{1}$,
$\DihedralAngle{1}{2}$
and $\DihedralAngle{2}{3}$.

Kellerhals used the {\em Schl\"afli differential formula}\/
to obtain these volume formulae.
The volume formulae
are not used directly
in our arguments;
what we will use
is the fact that the formulae
are parametrized
by
the
three
non-right
hyperbolic dihedral
angles.
On the other hand,
the Schl\"afli differential formula itself
plays
an important role in our arguments.

The Schl\"afli differential formula
gives an expression
of the differential form
of the volume function
with respect to the
hyperbolic lengths
of edges
and hyperbolic dihedral angles between faces.
As is given in Theorem~I in \cite{ke},
the differential form $d V$
of the volume function $V$ of 
any complete orthoschemes
is expressed as
\begin{equation*}
	d V
	= - \frac{1}{2} \left( \EdgeLength{0}{1} \, d \DihedralAngle{0}{1}
	+ \EdgeLength{1}{2} \, d \DihedralAngle{1}{2}
	+ \EdgeLength{2}{3} \, d \DihedralAngle{2}{3} \right) ,
\end{equation*}
where
$\EdgeLength{i}{j}$
is the hyperbolic length of the edge $\Edge{i}{j}$
if
both
$\Vertex{i}$ and $\Vertex{j}$ are points in $\Ball$,
$\EdgeLength{i}{j}$
is the hyperbolic distance
between $\Vertex{i}$
and the polar geodesic plane of $\Vertex{j}$
if
$\Vertex{i}$ is a point in $\Ball$
and $\Vertex{j}$ lies in the exterior of $\Ball$,
and
$\EdgeLength{i}{j}$
is the hyperbolic distance
between
the polar geodesic planes of
$\Vertex{i}$
and 
$\Vertex{j}$
if
both
$\Vertex{i}$ and $\Vertex{j}$ lie in the exterior of $\Ball$.
If a complete orthoscheme
is a Lambert cube,
then 
$\EdgeLength{0}{3}$
is taken as
the hyperbolic length
of the edge
obtained as the intersection
of the polar geodesic planes
of $\Vertex{0}$ and $\Vertex{3}$.
If 
one of
$\Vertex{0}$ and $\Vertex{3}$ is ideal,
then the edges with the ideal vertex as an endpoint
have infinite hyperbolic lengths.
In this case
we
take
any horosphere centered at
the ideal vertex,
and 
each infinite length
is replaced by
the signed
hyperbolic
distance
between the other endpoint and
the horosphere.
As is mentioned in the concluding remarks
in \cite{mi},
the Schl\"afli differential
formula
is
still
valid by this treatment.

As a result,
the Schl\"afli differential formula is applicable
to any kind of complete orthoschemes.
We use the formula
as the equation
\begin{equation} \label{eq: SchlafliFormula}
	\frac{\partial V}{\partial \DihedralAngle{i}{j}} = - \frac{1}{2} \, \EdgeLength{i}{j}
\end{equation}
for $(i , j) = (0 , 1) , (1 , 2) , (2 , 3)$.
This equation plays
a key role in the proof of Theorem~\ref{thm: main}.

\section{Main result} \label{sec: MainResult}
Suppose $\Ball$
lies in the $xyz$-coordinate space
of $\mathbb{R}^3$.
By the action of an isometry,
any ordinary orthoscheme
can be put as
the vertex $\Vertex{0}$
is in the positive quadrant of the $xy$-plane,
the vertex $\Vertex{1}$
is on the positive part of the $y$-axis,
the vertex $\Vertex{2}$ is the origin, and
the vertex $\Vertex{3}$ is
on the positive part of the $z$-axis.
Such orthoschemes
are parametrized by $(h , r , \theta)$,
where
$h$ is the 
$z$-coordinate of $\Vertex{3}$,
i.e.,
the Euclidean distance
between $\Vertex{2}$ and $\Vertex{3}$,
$r$ is the Euclidean distance
between $\Vertex{0}$ and $\Vertex{2}$,
and $\theta$
is the Euclidean angle between
edges $\Edge{0}{2}$ and $\Edge{1}{2}$.

When we regard such an orthoscheme
as a tetrahedron with base $\Face{0}{1}{2}$,
the $z$-coordinate $h$ of $\Vertex{3}$
is the ``height" of the tetrahedron.
What we study in this paper
is a family of complete orthoschemes
parametrized by the ``height".
For fixed $r$ and $\theta$,
we have a one-parameter family
$\SetS{ \Orthoscheme{h} }_{0 < h \leq 1}$
of ordinary orthoschemes parametrized by $h$.
This family is extended even when $h \geq 1$
and/or $r \geq 1$ with $r \cos \theta < 1$,
if we mean
$\Orthoscheme{h}$
a
complete orthoscheme.

Let $\VolumeFunction (h)$ be
the hyperbolic volume of
$\Orthoscheme{h}$.
By the volume formulae,
the function
$\VolumeFunction$
is continuous on $[0, + \infty)$
and piecewise differentiable
on the intervals
each of which
corresponds
to a combinatorial type of complete orthoschemes
given at the end of Section~\ref{sec: orthoschemes}.
When
$h$ increases in value approaching $1$,
the orthoscheme also increases
as a set.
So $\VolumeFunction (h)$ strictly increases in value
approaching $\VolumeFunction (1)$ as $h$ approaches $1$ from below.
When
$h$ approaches positive infinity $+\infty$,
the sequence $\Orthoscheme{h}$
of complete orthoschemes
converges to the base $\Face{0}{1}{2}$;
the complete orthoschemes are always ordinary ones
when
$0 < r \leq 1$,
and
the complete orthoschemes
changes into Lambert cubes from double frustums
when
$r > 1$.
In any case $\VolumeFunction (h)$ converges to $0$ as
$h$ approaches $+\infty$.

Based on these observations,
we have set the following questions.
For
a given one-parameter
family $\SetS{ \Orthoscheme{h} }_{h > 0}$
of complete orthoschemes,
does the function $\VolumeFunction$ attain maximal
when $\Vertex{3}$ is in $\partial \Ball$?
Is
$\VolumeFunction$ strictly decreasing on $(1 , +\infty)$?
The
next
theorem,
which is the main result of this paper,
answers both of the questions negatively.
%
%
%
%
%
%
\begin{thm} \label{thm: main}
	For any $r >0$ and $0 < \theta < \pi / 2$
	with $r \cos \theta < 1$,
	the volume $\VolumeFunction (h)$
	of $\Orthoscheme{h}$
	attains maximal for some $h \in (1 , + \infty)$.
	Furthermore,
	the maximal volume
	is unique
	for any $r$ and $\theta$,
	and it is given 
	before $\Orthoscheme{h}$ becomes
	a Lambert cube.
\end{thm}

The outline of the proof is as
follows.
Using
the Schl\"afli differential formula,
we can calculate
$d \VolumeFunction (h) / d h$
for each combinatorial types of  $\Orthoscheme{h}$.
Since $\VolumeFunction$
is a strictly increasing function on $[0,1]$,
proving
$\lim_{h \downarrow 1} d \VolumeFunction (h) / d h > 0$
tells us that
the function $\VolumeFunction$
attains maximal for some
$h \in (1 , + \infty)$.
The
uniqueness
of such $h$
is
induced
from the uniqueness
of the solution 
of the equation
$d \VolumeFunction (h) / d h = 0$
on $(1 , +\infty)$.

\section{Proof of the main result}
Our proof of Theorem~\ref{thm: main}
is organized as
follows.
After
confirming
the correspondence between combinatorial types of complete orthoschemes
and conditions of parameters $h$, $r$ and $\theta$,
we first obtain
suitable inverse images of
vertices
of $\Orthoscheme{h}$
for $\Projection$.
These are used
to calculate hyperbolic lengths and hyperbolic dihedral angles
appearing in the Schl\"afli differential formula.
Under each of
conditions of parameters,
we prove that
the volume function $\VolumeFunction$
with respect to $h$
attains maximal on $(1 , + \infty)$,
and that
such $h$ is unique.
For $r > 1$,
we also prove that
$\VolumeFunction$
does not attain maximal
if $\Orthoscheme{h}$ is a Lambert cube.

\subsection{Proper inverse images of the vertices}
By the definition of $\Orthoscheme{h}$,
the coordinates of the vertices
are
\begin{align*}
	\Vertex{0} &= (r \sin \theta , r \cos \theta , 0) ,&
	\Vertex{1} &= (0 , r \cos \theta , 0) ,\\
	\Vertex{2} &= (0 , 0 , 0) ,&
	\Vertex{3} &= (0 , 0 , h) ,
\end{align*}
where  $0 < \theta < \pi/2$.
As is mentioned after Theorem~\ref{thm: main},
it is enough to assume that $h > 1$ in what follows.
A complete orthoscheme $\Orthoscheme{h}$ is
a simple frustum if $0 < r < 1$, and
a simple frustum with ideal vertex $\Vertex{0}$ if $r = 1$.
When $r > 1$,
we always assume
$r \cos \theta < 1$
so that $\Vertex{1}$ is in $\Ball$.
Under these assumptions,
a complete orthoscheme $\Orthoscheme {h}$ with $r>1$
is
either a double frustum,
a double frustum with an ideal vertex, or
a Lambert cube.
These are distinguished via
the Euclidean distance between
the origin of $\mathbb{R}^3$
and the edge $\Edge{0}{3}$;
$\Orthoscheme {h}$ is a double frustum,
a double frustum with an ideal vertex, or a Lambert cube
if and only if the Euclidean distance is less than,
equal to, or greater than $1$ respectively.
Since the Euclidean distance
is $h \, r / \sqrt{r^2 + h^2}$,
we have that
these are equivalent to
$h < r / \sqrt{r^2 - 1}$,
$h = r / \sqrt{r^2 - 1}$, or
$h > r / \sqrt{r^2 - 1}$ respectively.
The inequality  
$h \, r / \sqrt{r^2 + h^2} < 1$
is also equivalent to 
$\left( 1 - r^2 \right) h^2 + r^2 > 0$
without the assumption that $r > 1$.
We note that
this inequality always holds for any $h > 0$ and $0 < r \leq 1$.

As a summary,
complete orthoschemes $\Orthoscheme{h}$
are parametrized by $(h,r,\theta)$, and
with
$h > 1$ and $0 < \theta < \pi/2$, and
\begin{itemize}
	\item
		when $0 < r \leq 1$,
		complete orthoschemes $\Orthoscheme{h}$ are simple frustums
		with $\left( 1 - r^2 \right) h^2 + r^2 > 0$,
	\item
		when $r > 1$ with
		$r \cos \theta < 1$ and $h \leq r / \sqrt{r^2 - 1}$,
		complete orthoschemes
		$\Orthoscheme{h}$ are double frustums (possibly with an ideal vertex), and
	\item
		when $r > 1$ with
		$r \cos \theta < 1$ and $h > r / \sqrt{r^2 - 1}$,
		complete orthoschemes $\Orthoscheme{h}$ are Lambert cubes.
\end{itemize}

We next
give the proper inverse images of
these vertices for $\Projection$.
When a vertex is in $\Ball$,
its inverse image for $\Projection$
must be chosen in $\HyperboloidUpper$,
which is uniquely determined.
When a vertex is in the exterior of $\Ball$,
its inverse image is chosen
to be proper inverse image
in the sense of
truncation.
Finally, when a vertex is in $\partial \Ball$,
we choose its
proper inverse image
as any element in the inverse for $\Projection$,
which is a subset in $\LightCone$.
Let $\VertexLift{i}$
be the proper inverse image of $\Vertex{i}$
in this sense.
The coordinates of $\VertexLift{i}$
are
then
as
follows:
\begin{enumerate}
	\item \label{enu: case r<1}
		When $0 < r < 1$,
		we have
		\begin{align*}
			\VertexLift{0}
			&= \frac{1}{\sqrt{1 - r^2}} \left(1, r \sin \theta , r \cos \theta , 0 \right) ,\\
			\VertexLift{1}
			&= \frac{1}{\sqrt{1 - r^2 \cos^2 \theta}} \left(1, 0 , r \cos \theta , 0 \right) ,\\
			\VertexLift{2}
			&= (1,0,0,0) ,\\
			\VertexLift{3}
			&= \frac{1}{\sqrt{ h^2 - 1}} \left(1, 0 , 0 , h \right) .
		\end{align*}
	\item \label{enu: case r=1}
		When $r=1$,
		the
		coordinates of $\VertexLift{1}$, $\VertexLift{2}$ and $\VertexLift{3}$
		are the same as in the first case and
		\begin{equation*}
			\VertexLift{0}
			= \left(1, \sin \theta , \cos \theta , 0 \right) .
		\end{equation*}
	\item \label{enu: case r>1}
		When $r > 1$,
		the
		coordinates of $\VertexLift{1}$, $\VertexLift{2}$ and $\VertexLift{3}$
		are the same as in the first case, and
		\begin{equation*}
			\VertexLift{0}
			=\frac{1}{\sqrt{r^2 - 1}} \left(1, r \sin \theta , r \cos \theta , 0 \right) .
		\end{equation*}
\end{enumerate}
The inverse image
of the pole of a geodesic plane
in $\Ball$
consists of two points in $\HyperboloidOne$.
For each (ordinary) face of
an orthoscheme $\Orthoscheme{h}$,
we choose the inverse image
of the pole
in $\HyperboloidOne$
so that the half-space
defined by this inverse image contains $\Orthoscheme{h}$.
Let $\FaceNormal{i}$
be the inverse image of the pole of
the face $\Face{j}{k}{l}$
for $\SetS{i , j , k , l} = \SetS{0 , 1 , 2 , 3}$
in this sense.
In other words,
$\FaceNormal{i}$ is a point in $\HyperboloidOne$
where
$\HalfSpace{\FaceNormal{i}}$ contains $\Orthoscheme{h}$
and
$\Plane{\FaceNormal{i}}$
contains $\Face{j}{k}{l}$.
For any $r$,
the coordinates of $\FaceNormal{i}$
are as follows:
\begin{align*}
	\FaceNormal{0}
	&= \left(0,-1,0,0 \right) ,\\
	\FaceNormal{1}
	&= \left(0, \cos \theta ,  - \sin \theta , 0 \right) ,\\
	\FaceNormal{2}
	&= \frac{1}{\sqrt{\left( 1 - r^2 \cos^2 \theta \right) h^2 + r^2 \, \cos^2 \theta}} \left( h \, r \cos \theta , 0 , h  , r \cos \theta \right) ,\\
	\FaceNormal{3}
	&= \left(0, 0 , 0 , -1 \right) .
\end{align*}

\subsection{The maximal value of $\VolumeFunction$ and its uniqueness with respect to $h$} \label{subsec: MaxValue}
We focus on the derivative
$d \VolumeFunction (h) / d h$
to prove that
$\VolumeFunction$
attains maximal
on $(1,+ \infty)$,
as well as
its uniqueness.

We first confirm that
the function
$\VolumeFunction$
is piecewise differentiable with respect to $h$ in general.

We first suppose that $r \leq 1$.
By the Schl\"afli differential formula,
the function $\VolumeFunction$
is differentiable with respect to the hyperbolic dihedral angles
$\DihedralAngle{0}{1}$,
$\DihedralAngle{1}{2}$ and
$\DihedralAngle{2}{3}$.
By the expression of the coordinates of $\FaceNormal{i}$ and $\VertexLift{i}$
for $i = 0 , 1 , 2 , 3$
together with Equation~\eqref{eq: DihedralAngle},
these angles are given as smooth functions with respect to $h$.
By the chain rule,
$\VolumeFunction$
is thus differentiable with respect to $h$.
In particular 
$\VolumeFunction$
is continuous on $[0,+ \infty)$.

If $r > 1$,
then there are two combinatorial types
of $\Orthoscheme{h}$;
a double frustum or a Lambert cube.
The function
$\VolumeFunction$
is not only continuous but also
piecewise differentiable on $[0, + \infty)$,
for $\VolumeFunction$
is differentiable
on the intervals corresponding to
each combinatorial types of 
$\Orthoscheme{h}$
by the same argument used for $r \leq 1$.

Recall that
the function
$\VolumeFunction$ 
is continuous
on $[0,+ \infty)$,
strictly increasing on $[0,1]$
and
has its limit $0$ as $h$ approaches $+ \infty$. 
So,
to prove that 
$\VolumeFunction$
attains
maximal
on $(1 , + \infty)$,
it is enough to prove that 
the limit of $d \VolumeFunction (h) / d h$
is positive
as $h$ approaches to $1$ from above.
The uniqueness
of the maximal value of 
$\VolumeFunction$
is induced from the fact
that the solution
of
$d \VolumeFunction (h) / d h = 0$
is at most one
on $(1 , + \infty)$.

Applying the chain rule and
we have
\begin{equation*}
	\frac{d \VolumeFunction (h)}{d h}
	= \frac{\partial \VolumeFunction (h)}{\partial \DihedralAngle{0}{1}} \frac{d \DihedralAngle{0}{1}}{d h}
	+ \frac{\partial \VolumeFunction (h)}{\partial \DihedralAngle{1}{2}} \frac{d \DihedralAngle{1}{2}}{d h}
	+ \frac{\partial \VolumeFunction (h)}{\partial \DihedralAngle{2}{3}} \frac{d \DihedralAngle{2}{3}}{d h} .
\end{equation*}
The parameter
$\DihedralAngle{i}{j}$
defined in Section~\ref{sec: Schlafli}
is the hyperbolic dihedral angle between 
the polar geodesic planes of $\Projection (\FaceNormal{i})$
and 
$\Projection (\FaceNormal{j})$.
In other words,
$\DihedralAngle{i}{j}$ is
the hyperbolic dihedral angle along the edge $\Edge{k}{l}$
for $\SetS{i , j , k , l} = \SetS{0 , 1 , 2 , 3}$.
As is mentioned in the first paragraph of Section~\ref{sec: Schlafli},
if $\Orthoscheme {h}$ is a Lambert cube,
then $\DihedralAngle{1}{2}$
is taken as the hyperbolic dihedral angle
between
the polar geodesic planes of $\Vertex{0}$ and $\Vertex{3}$.

The Schl\"afli differential formula are used to calculate
partial derivatives appeared in the equation above.
By Equation~\eqref{eq: SchlafliFormula} we have
\begin{align*}
	\frac{\partial \VolumeFunction (h)}{\partial \DihedralAngle{1}{2}} &= - \frac{1}{2} \, \EdgeLength{0}{3} ,&
	\frac{\partial \VolumeFunction (h)}{\partial \DihedralAngle{2}{3}} &= - \frac{1}{2} \, \EdgeLength{0}{1} ,
\end{align*}
where $\EdgeLength{i}{j}$
is the hyperbolic length with respect to
the edge $\Edge{i}{j}$
defined in Section~\ref{sec: Schlafli}.
Furthermore,
the hyperbolic dihedral angle $\DihedralAngle{0}{1}$,
which coincides with the Euclidean angle $\theta$
by the definition of $\Orthoscheme{h}$,
is constant with respect to $h$,
meaning that $ d \DihedralAngle{0}{1} / d h = 0$.
We thus have
\begin{equation} \label{eq: dvdh}
	\frac{d \VolumeFunction (h)}{d h}
	=
	- \frac{1}{2}
	\left( \EdgeLength{0}{3} \frac{d \DihedralAngle{1}{2}}{d h} + \EdgeLength{0}{1} \frac{d \DihedralAngle{2}{3}}{d h} \right) .
\end{equation}

We divide the remaining argument
into three cases according to the value of $r$.

\subsubsection*{Case~\eqref{enu: case r<1}: single frustums with ordinary vertex $\Vertex{0}$, i.e., $0 < r < 1$} 
By Equations~\eqref{eq: DistancePointPlane}
we have
\begin{align}
	\EdgeLength{0}{3} \nonumber
	&= \Arcsinh \left( - \LIP{\VertexLift{0}}{\VertexLift{3}} \right)  \nonumber\\
	&= \Arcsinh \frac{1}{\sqrt{1-r^2} \, \sqrt{h^2-1}}  \nonumber\\
	&= \log \left( \frac{1}{\sqrt{1-r^2} \, \sqrt{h^2-1}}
		+ \sqrt{ \left( \frac{1}{\sqrt{1-r^2} \, \sqrt{h^2-1}} \right)^2 + 1} \right)  \nonumber\\
	&= \log \frac{\sqrt{\left(1 - r^2 \right) h^2 + r^2} + 1}{\sqrt{1-r^2} \, \sqrt{h^2-1}} , \label{eq: EdgeLength03}
\end{align}
and by Equation~\eqref{eq: DihedralAngle}
we have
\begin{align}
	\DihedralAngle{1}{2} \nonumber
	&= \arccos \left( - \LIP{\FaceNormal{1}}{\FaceNormal{2}} \right) \nonumber\\
	&= \arccos \frac{h \sin \theta}{\sqrt{ \left(1 - r^2 \cos^2 \theta \right) h^2 + r^2 \cos^2 \theta}} , \nonumber\\
	\DihedralAngle{2}{3} \nonumber
	&= \arccos \left( - \LIP{\FaceNormal{2}}{\FaceNormal{3}} \right) \nonumber\\
	&= \arccos \frac{r \cos \theta}{\sqrt{ \left( 1 - r^2 \cos^2 \theta \right) h^2 + r^2 \cos^2 \theta }} . \nonumber
\end{align}
Derivatives of hyperbolic dihedral angles with respect to $h$ are obtained as follows:
\begin{align}
	\label{eq: DerT12}
	\frac{d \DihedralAngle{1}{2}}{d h}
	&= \frac{ - r^2 \sin \theta \cos \theta}{\left\{ \left( 1 - r^2 \cos^2 \theta \right) h^2
		+ r^2 \cos^2 \theta \right\} \sqrt{\left( 1 - r^2 \right) h^2 + r^2}} ,\\
	\label{eq: DerT23}
	\frac{d \DihedralAngle{2}{3}}{d h}
	&= \frac{ r \sqrt{1 - r^2 \cos^2 \theta} \, \cos \theta }{\left( 1 - r^2 \cos^2 \theta \right) h^2 + r^2 \cos^2 \theta} .
\end{align}
Substitute Equations~\eqref{eq: EdgeLength03}, \eqref{eq: DerT12} and \eqref{eq: DerT23}
to Equation~\eqref{eq: dvdh}
and we have
\begin{equation*}
	\frac{d \VolumeFunction (h)}{d h}
	=
	\frac{1}{2} \left( - \frac{d \DihedralAngle{1}{2}}{d h} \right) \left( F(h) - \frac{1}{2} \, \log (1 - r^2) \right) ,
\end{equation*}
where
\begin{equation} \label{eq: DefOfF}
	F(h)
	\vcentcolon =
	\log \frac{\sqrt{\left(1 - r^2 \right) h^2 + r^2} + 1}{ \sqrt{h^2-1} }
	- C \, \sqrt{ \left(1 - r^2 \right) h^2 + r^2 }
\end{equation}
and $C := \EdgeLength{0}{1} \, \sqrt{1 - r^2 \cos^2 \theta} / \! \left( r \sin \theta \right)$.

Since
\begin{align*}
	\lim_{h \downarrow 1}
	\left( - \frac{d \DihedralAngle{1}{2}}{d h} \right)
	&= r^2 \sin \theta \cos \theta ,&
	\lim_{h \downarrow 1}
	F (h) 
	&= + \infty ,
\end{align*}
we have
\begin{align*}
	\lim_{h \downarrow 1} \frac{d \VolumeFunction (h)}{d h}
	&=
	\frac{1}{2}
	\left( r^2 \sin \theta \cos \theta \right)
	\left( + \infty - \frac{1}{2} \, \log (1 - r^2 ) \right)\\
	&= + \infty ,
\end{align*}
which implies that
$\VolumeFunction$
attains maximal
for some
$h \in (1 , + \infty)$.

This result
together with
$\lim_{h \uparrow + \infty} \VolumeFunction (h) = 0$
implies
that
the uniqueness
of the maximal value of
the function $\VolumeFunction$
with respect to $h$
is proved
by showing
that
the equation
$d \VolumeFunction (h) / d h = 0$
has at most one solution on
$(1 , + \infty)$.
Since
$d \DihedralAngle{1}{2} / d h \ne 0$
on
$(1 , + \infty)$
by Equation~\eqref{eq: DerT12},
we have
\begin{multline} \label{eq: dfmost1}
	\Set{ h \in ( 1 , + \infty ) | \frac{d \VolumeFunction (h)}{d h} = 0 } \\
	=
	\Set{ h \in ( 1 , + \infty ) | F(h) - \frac{1}{2} \, \log (1 - r^2) = 0 } .
\end{multline}
Since
\begin{equation} \label{eq: dfdh}
	\frac{d}{d h} \left( F (h)  - \frac{1}{2} \, \log (1 - r^2) \right)
	= - \frac{ h }{\left( h^2-1 \right) \sqrt{ \left(1 - r^2 \right) h^2 + r^2}} \, G(h),
\end{equation}
where
$G(h) \vcentcolon = C \left( 1-r^2 \right) \left( h^2-1 \right) + 1$,
is negative on $(1,+\infty)$,
the function
$F(h) - \left( 1/2 \right) \log (1 - r^2)$
is strictly monotonic with respect to $h$.
This implies that
the number of elements
in the right-hand side set of Equation~\eqref{eq: dfmost1}
is at most one,
so is the left-hand side.

\subsubsection*{Case~\eqref{enu: case r=1}: single frustums with ideal vertex $\Vertex{0}$, i.e., $r = 1$}
Using Equation~\eqref{eq: DistancePointHorosphere},
we have
\begin{align*}
	\EdgeLength{0}{3}
	&= \log \left( - 2 \LIP{\VertexLift{0}}{\VertexLift{3}} \right) \\
	&= \log \frac{2}{\sqrt{h^2 - 1}} .
\end{align*}
By Equations~\eqref{eq: DerT12} and \eqref{eq: DerT23} with $r=1$
and we have
\begin{equation*}
	- \frac{d \DihedralAngle{1}{2}}{d h}
	= \frac{d \DihedralAngle{2}{3}}{d h}
	= \frac{ \sin \theta \cos \theta}{h^2 \sin^2 \theta + \cos^2 \theta} .
\end{equation*}
Substitute
these equations
to Equation~\eqref{eq: dvdh} and we have
\begin{align*}
	\frac{d \VolumeFunction (h)}{d h}
	&=
	\frac{1}{2} \left( - \frac{d \DihedralAngle{1}{2}}{d h} \right)
	\left( \log \frac{2}{\sqrt{h^2-1}} - \EdgeLength{0}{1} \right) \\
	&=
	\frac{1}{2} \left( - \frac{d \DihedralAngle{1}{2}}{d h} \right)
	\left( - \frac{1}{2} \log (h^2 - 1 ) + \log 2 - \EdgeLength{0}{1} \right)	.
\end{align*}

Since
\begin{align*}
	\lim_{h \downarrow 1}
	\left( - \frac{d \DihedralAngle{1}{2}}{d h} \right)
	&= \sin \theta \cos \theta ,&
	\lim_{h \downarrow 1}
	\log ( h^2 - 1 )
	&= - \infty ,
\end{align*}
we have
$\lim_{h \downarrow 1} d \VolumeFunction (h) / d h = + \infty$ in this case.

The uniqueness of the maximal value of $\VolumeFunction$
with respect to $h$
is obtained by
the facts that $\log ( h^2 - 1 )$
is a strictly monotonic function
and that 
 $d \DihedralAngle{1}{2} / d h \ne 0$
on $(1,+\infty)$.

\subsubsection*{Case~\eqref{enu: case r>1}: double frustums or Lambert cubes, i.e., $r > 1$}
Since our strategy
of proving that
$\VolumeFunction$
attains maximal
on $(1 , + \infty)$
is to prove that
the limit of $d \VolumeFunction (h) / d h$
is positive
as $h$ approaches to $1$ from above,
it is enough to consider the case
that $h$ is close enough to $1$,
meaning that 
$\Orthoscheme{h}$ are double frustum, not Lambert cubes.

Under this assumption,
use Equation~\eqref{eq: DistancePlanes} and we have
\begin{align*}
	\EdgeLength{0}{3}
	&= \Arccosh \left( - \LIP{\VertexLift{0}}{\VertexLift{3}} \right) \\
	&= \Arccosh \frac{1}{\sqrt{r^2 - 1} \, \sqrt{h^2 - 1}} \\
	&= \log \left( \frac{1}{\sqrt{r^2 - 1} \, \sqrt{h^2 - 1}}
		+ \sqrt{ \left( \frac{1}{\sqrt{r^2 - 1} \, \sqrt{h^2 - 1}} \right)^2 - 1} \right) \\
	&= \log \frac{\sqrt{\left(1 - r^2 \right) h^2 + r^2} + 1}{\sqrt{r^2 - 1} \, \sqrt{h^2 - 1}} .
\end{align*}
Substitute this equation
together with Equations~\eqref{eq: DerT12} and \eqref{eq: DerT23}
to Equation~\eqref{eq: dvdh} and we have
\begin{equation} \label{eq: dvdhr>1}
	\frac{d \VolumeFunction (h)}{d h}
	=
	\frac{1}{2} \left( - \frac{d \DihedralAngle{1}{2}}{d h} \right)
	\left( F(h) - \frac{1}{2} \, \log (r^2 - 1) \right) ,
\end{equation}
where $F$ is the function defined in Case~\eqref{enu: case r<1}.

By the same reason explained in Case~\eqref{enu: case r<1},
we have 
$\lim_{h \downarrow 1} d \VolumeFunction (h) / d h = + \infty$ in this case as well.

\medskip

We next prove that
$\VolumeFunction$
does not attain maximal
when $\Orthoscheme{h}$
is a Lambert cube,
i.e.,
$h \in (r / \sqrt{r^2 - 1} , + \infty)$.
What we actually prove
is that
$\VolumeFunction$
is strictly decreasing,
using Equation~\eqref{eq: dvdh}.
Recall
that
$\EdgeLength{0}{3}$ is
the hyperbolic distance
between the polar geodesic plane of $\Projection (\FaceNormal{1})$
and $\Projection (\FaceNormal{2})$,
and
$\DihedralAngle{1}{2}$
is the hyperbolic dihedral angle
between the polar geodesic planes of 
$\Vertex{0}$ and $\Vertex{3}$,
while $\EdgeLength{0}{1}$ and $\DihedralAngle{2}{3}$ are the same as in other cases.
Using Equations~\eqref{eq: DistancePlanes} and \eqref{eq: DihedralAngle},
we have
\begin{align*}
	\EdgeLength{0}{3}
	&= \Arccosh \left( - \LIP{\FaceNormal{1}}{\FaceNormal{2}} \right) \\
	&= \Arccosh \frac{h \sin \theta}{\sqrt{ \left( 1 - r^2 \cos^2 \theta \right) h^2 + r^2 \cos^2 \theta }} \\
	&= \log \frac{ h \sin \theta + \sqrt{ \left( r^2 - 1 \right) h^2 - r^2 } \, \cos \theta }{\sqrt{ \left( 1 - r^2 \cos^2 \theta \right) h^2 + r^2 \cos^2 \theta }} ,\\
	\DihedralAngle{1}{2}
	&= \arccos \left( - \LIP{\VertexLift{0}}{\VertexLift{3}} \right) \\
	&= \arccos \frac{1}{\sqrt{r^2 - 1} \, \sqrt{h^2 - 1}} ,\\
	\frac{d \DihedralAngle{1}{2}}{d h}
	&= \frac{h}{\left( h^2 - 1 \right) \sqrt{ \left( r^2 - 1 \right) h^2 - r^2 }} .
\end{align*}
The value 
$d \DihedralAngle{1}{2} / d h$
is positive
on $(r / \sqrt{r^2 - 1} , + \infty)$
by this expression,
so is $d \DihedralAngle{2}{3} / d h$ by Equation~\eqref{eq: DerT23}.
The value $\EdgeLength{0}{1}$ is positive, for it is the hyperbolic length of an edge.
By substituting these results to Equation~\eqref{eq: dvdh},
if we can prove that
$\EdgeLength{0}{3} > 0$,
then we have $d \VolumeFunction / d h < 0$,
namely 
$\VolumeFunction$
is strictly decreasing,
on $(r / \sqrt{r^2 - 1} , + \infty)$.

The inequality 
$\EdgeLength{0}{3} > 0$
is equivalent to
\begin{equation*}
	\frac{ h \sin \theta +
		\sqrt{ \left( r^2 - 1 \right) h^2 - r^2 } \, \cos \theta }{\sqrt{ \left( 1 - r^2 \cos^2 \theta \right) h^2 + r^2 \cos^2 \theta }}
	> 1 .
\end{equation*}
Calculating
\begin{equation*}
	\left( \frac{ h \sin \theta
		+ \sqrt{ \left( r^2 - 1 \right) h^2 - r^2 } \, \cos \theta }{\sqrt{ \left( 1 - r^2 \cos^2 \theta \right) h^2
			+ r^2 \cos^2 \theta }} \right)^2 - 1
\end{equation*}
and 
we have an inequality
\begin{equation*}
	\sqrt{\left( r^2 - 1 \right) h^2 - r^2} \, h \sin \theta
	> - \left\{ \left( r^2 - 1 \right) h^2 - r^2 \right\} \cos \theta ,
\end{equation*}
which is equivalent to the previous one.
This inequality holds on $(r / \sqrt{r^2 - 1} , + \infty)$,
for the right-hand side is negative
while the left hand side is positive.
We have thus proved
that
$\VolumeFunction$
does not attain maximal when $\Orthoscheme{h}$ is a Lambert cube.

\medskip

Since $\VolumeFunction$
does not attain maximal
when $\Orthoscheme{h}$
is a Lambert cube,
for the proof of
the
uniqueness of the maximal value
of $\VolumeFunction$,
we can assume
that
$h \in (1 , r / \sqrt{r^2 - 1} ]$.
Under this assumption
together with
the fact that
$d \DihedralAngle{1}{2} / d h \ne 0$
on
$(1 , r / \sqrt{r^2 - 1})$
by Equation~\eqref{eq: DerT12},
what we need to prove is that
the number of elements in the set
\begin{multline*}
	\Set{ h \in ( 1 , \frac{r}{\sqrt{r^2 - 1}} ) | \frac{d \VolumeFunction (h)}{d h} = 0 } \\
	=
	\Set{ h \in ( 1 , \frac{r}{\sqrt{r^2 - 1}} ) | F(h) - \frac{1}{2} \, \log (r^2 - 1) = 0 }
\end{multline*}
is at most one,
where
$d \VolumeFunction (h) / d h$
is calculated in Equation~\eqref{eq: dvdhr>1}
and
the function $F$ is given in
Equation~\eqref{eq: DefOfF}.

By Equation~\eqref{eq: dfdh},
we have
\begin{equation*}
	\frac{d}{d h} \left( F (h)  - \frac{1}{2} \, \log ( r^2 - 1) \right)
	= - \frac{ h }{\left( h^2 - 1 \right) \sqrt{  \left(1 - r^2 \right) h^2 + r^2 }} \, G(h) ,
\end{equation*}
where we recall that $G(h) = C \left( 1 - r^2 \right) \left( h^2 - 1 \right) + 1 $.
Unlike Case~\eqref{enu: case r<1},
the sign of 
the function $G$
is not expected to be constant on $(1 , r / \sqrt{r^2 - 1})$,
for $1 - r^2 < 0$.

Since
\begin{equation*}
	\frac{ h }{\left( h^2 - 1 \right) \sqrt{  \left(1 - r^2 \right) h^2 + r^2 }} \ne 0
\end{equation*}
on $(1 , r / \sqrt{r^2 - 1})$,
we have
\begin{equation*}
	\Set{ h \in ( 1 , \frac{r}{\sqrt{r^2 - 1}} ) | F'(h) = 0 }
	=
	\Set{ h \in ( 1 , \frac{r}{\sqrt{r^2 - 1}} ) | G(h) = 0 } .
\end{equation*}
The function $G$ is quadratic with respect to $h$,
the coefficient of $h^2$ is negative and $G(1) > 0$.
These imply that
the number of elements in
the set of the right-hand side of the equation above
is at most one,
so is the set of the left-hand side of the equation.

Suppose that the number of elements
in the set 
\begin{equation*}
	\Set{ h \in ( 1 , \frac{r}{\sqrt{r^2 - 1}} ) | F(h) - \frac{1}{2} \, \log (r^2 - 1) = 0 }
\end{equation*}
is more than $1$.
By the mean-value theorem
together with the fact that
the limit
of
$F(h) - \left( 1/2 \right) \log (r^2 - 1)$
is
$0$
as $h$ approaches $r / \sqrt{r^2 - 1}$ from below,
the set
$\Set{ h \in ( 1 , r / \sqrt{r^2 - 1} ) | F'(h) = 0 }$
must contain at least two elements,
which contradicts the result obtained above.

\bigskip

We have thus proved
Theorem~\ref{thm: main}.
\hfill $\Box$

\appendix
\section{The maximal area of two-dimensional hyperbolic complete orthoschemes}

By the definition of orthoscheme,
a triangle $\Face{0}{1}{2}$ in the two-dimensional hyperbolic space
is orthoscheme
if
the edge $\Edge{0}{1}$ is perpendicular to the edge $\Edge{1}{2}$,
namely $\Face{0}{1}{2}$ is a right-angled triangle with the right angle at $\Vertex{1}$.
Without loss of generality,
we suppose that $\Face{0}{1}{2}$ lies in the projective disc model $\Disc$
with the coordinates
\begin{align*}
	\Vertex{0} &= (r,0) ,&
	\Vertex{1} &= (0,0) ,&
	\Vertex{2} &= (0,h) .
\end{align*}
For a given $r > 0$,
we consider a family $\SetS{ \RAT{h} }_{h>0}$
of complete orthoschemes,
where $\RAT{h}$
is a complete orthoscheme with vertices $\Vertex{0}$, $\Vertex{1}$ and $\Vertex{2}$.
What we discuss
is the maximal area
for
this family.
\begin{thm} \label{thm: 2ortho}
	The maximal area for
	$\SetS{\RAT{h}}_{h>0}$
	is obtained as follows:
	\begin{enumerate}
		\item
			For any $r < 1$,
			the area of $\RAT{h}$
			attains maximal
			just for $h = 1$.
			The maximal area is
			$\pi / 2 - \AngleAtZero{1}$,
			where $\AngleAtZero{1}$ is
			the hyperbolic angle at $P_0$ of $\RAT{1}$.
		\item
			The area of $\RATT{1}{h}$
			attains maximal
			for any $h \in [1 + \infty )$.
			The maximal area is
			$\pi / 2$.
		\item
			For any $r > 1$,
			the area of $\RAT{h}$
			attains maximal
			for any $h \in [1 , r / \sqrt{r^2 - 1}]$.
			The maximal area is $\pi / 2$.
	\end{enumerate}
\end{thm}

\begin{proof}
We start by recalling a formula to calculate
the area $A$ of a hyperbolic convex $n$-gon with hyperbolic angles $\alpha_1 , \alpha_2 , \dotsc , \alpha_n$;
\begin{equation*}
	A = \left( n - 2 \right) \pi - \left( \alpha_1 + \alpha_2 + \dotsb + \alpha_n \right) .
\end{equation*}
See Theorem~3.5.5 of \cite{ra} for the proof when $n=3$.

Let
$\AreaFunction{h}$
be the area of $\RAT{h}$.
For any $r > 0$,
a complete orthoscheme
$\RAT{h}$ increases as a set
when $h$ approaches $1$ from below,
which implies that
$\AreaFunction{h}$
also increases.
So,
to prove the theorem,
it is enough to assume that $h \geq 1$.
Using this formula,
we obtain the area of $\RAT{h}$ for each
case.
\begin{enumerate}
\item
	Suppose $r<1$.
	Let $\AngleAtZero{h}$ be the hyperbolic angle
	at $\Vertex{0}$ of $\RAT{h}$.
	
	When $h = 1$,
	$\RAT{1}$ is a triangle
	with ideal vertex $\Vertex{2}$.
	Since the hyperbolic angle at $\Vertex{2}$ is $0$,
	the area is
	\begin{align*}
		\AreaFunction{1}
		&= \pi - \left( \AngleAtZero{1} + \frac{\pi}{2} + 0 \right) \\
		&= \frac{\pi}{2} - \AngleAtZero{1} .
	\end{align*}
	
	When $h > 1$,
	$\RAT{h}$ is a quadrilateral.
	The hyperbolic angles
	at the vertices constructed by truncation with respect to $\Vertex{2}$
	are right angles.
	The area is
	\begin{align*}
		\AreaFunction{h}
		&= 2 \, \pi - \left( \AngleAtZero{h} + \frac{\pi}{2} + \left( \frac{\pi}{2} + \frac{\pi}{2} \right) \right)\\
		&= \frac{\pi}{2} - \AngleAtZero{h} .
	\end{align*}
	When $h$
	approaches
	$+ \infty$,
	the corner at $\Vertex{0}$ increases as a set,
	so is the angle $\AngleAtZero{h}$.
	This implies that
	$\AreaFunction{h}$ is a strictly decrease
	function on $[1 , + \infty)$.
	
	As a result,
	$\AreaFunction{h}$ attains
	maximal
	if and only if
	$h=1$
	in this case.

\item
	Suppose $r=1$.
	The hyperbolic angle at $\Vertex{0}$ is $0$ in this case.
	Use the argument in (1) with $\AngleAtZero{h} = 0$
	for any $h \geq 1$
	and we have the desired conclusion.
	
\item
	Suppose $r > 1$.
	
	When $h=1$,
	$\RAT{1}$
	is a quadrilateral
	with angle $0$ at $\Vertex{2}$
	and three right angles.
	The area is
	\begin{align*}
		\AreaFunction{1}
		&= 2 \, \pi - \left( \left( \frac{\pi}{2} + \frac{\pi}{2} \right) + \frac{\pi}{2} + 0 \right)\\
		&= \frac{\pi}{2} .
	\end{align*}

	When $h>1$,
	there are two kinds of $\RAT{h}$,
	which correspond to
	double frustums and Lambert cubes of
	three-dimensional complete orthoschemes.
	\begin{itemize}
	\item
		If $h < r / \sqrt{r^2 - 1}$,
		then $\RAT{h}$ is a right-angled pentagon.
		The area is
		\begin{align*}
			\AreaFunction{h}
			&= 3 \, \pi - \frac{\pi}{2} \times 5\\
			&= \frac{\pi}{2} .
		\end{align*}
	\item
		If $h \geq r / \sqrt{r^2 - 1}$,
		then $\RAT{h}$ is a quadrilateral,
		whose edges consists of
		$\Edge{0}{1}$,
		$\Edge{1}{2}$ and
		polar lines of  $\Vertex{0}$ and $\Vertex{2}$.
		Let $b$ be
		the hyperbolic angle between
		these polar lines.
		Then the area is
		\begin{align*}
			\AreaFunction{h}
			&= 2 \, \pi - \left( \frac{\pi}{2} \times 3 + b \right)\\
			&= \frac{\pi}{2} - b .
		\end{align*}
		The maximal area
		arrises when $b=0$,
		which occurs if and only if the polar planes of
		$\Vertex{0}$ and $\Vertex{2}$ are parallel,
		namely $h = r / \sqrt{r^2 - 1}$.
	\end{itemize}
\end{enumerate}

Summarizing these results,
we have completed the proof.
\end{proof}

\section*{Acknowledgements}
The authors would like to thank
the referee for his/her careful reading and useful suggestions.

\end{document}